\documentclass[11pt]{article}

\usepackage[margin= 1.6 cm]{geometry}
\usepackage[utf8]{inputenc} 
\usepackage{amsmath}
\usepackage{graphicx}
\usepackage{subcaption}
\usepackage{amsfonts}
\usepackage{mathtools}
\usepackage{subcaption}
\usepackage{amssymb}
\usepackage{amsthm}
\usepackage{chngpage}
\graphicspath{ {Images/} }
\usepackage{hyperref}
\hypersetup{
	colorlinks=true,
	linkcolor=black,
	citecolor=black,
	filecolor=black,      
	urlcolor=black,
}
\usepackage{enumerate}
\usepackage{makeidx}
\usepackage{thmtools}
\usepackage{thm-restate}

\usepackage[dvipsnames]{xcolor}
\usepackage{tikz}
\usetikzlibrary{decorations.pathreplacing}

\makeindex

\makeatletter
\renewenvironment{proof}[1][\proofname] {\par\pushQED{\qed}\normalfont\topsep6\p@\@plus6\p@\relax\trivlist\item[\hskip\labelsep\bfseries#1\@addpunct{.}]\ignorespaces}{\popQED\endtrivlist\@endpefalse}
\makeatother

\newtheorem{proposition}{Proposition}[section]
\newtheorem{conjecture}[proposition]{Conjecture}
\newtheorem{lemma}[proposition]{Lemma}
\newtheorem{theorem}[proposition]{Theorem}

\theoremstyle{definition}

\newtheorem*{remark*}{Remark}
\newtheorem*{theorem*}{Theorem}
\newtheorem{claim}[proposition]{Claim}

\usepackage[shortlabels]{enumitem}
\usepackage[capitalise]{cleveref}
\usepackage{url}

\usepackage[numbers,sort]{natbib}

\newcommand{\eps}{\varepsilon}
\newcommand{\PP}{\mathbb{P}}
\newcommand{\EE}{\mathbb{E}}
\newcommand{\ex}{\operatorname{ex}}
\newcommand{\type}{\operatorname{type}}
\newcommand{\subtype}{\operatorname{subtype}}
\renewcommand{\epsilon}{\varepsilon}
\title{}
\author{Barnab\'as Janzer\thanks{Mathematical Institute, University of Oxford, United Kingdom. Research supported by a fellowship at Magdalen College, and by ERC grant No.~947978 while the author was affiliated with the University of Warwick. Email: \mbox{\textbf{barnabas.janzer@magd.ox.ac.uk}}.}
	\and
Oliver Janzer\thanks{Department of Pure Mathematics and Mathematical Statistics, University of Cambridge, United Kingdom. Research supported by a fellowship at Trinity College. Email: \textbf{oj224@cam.ac.uk}.}
\and
Van Magnan\thanks{Department of Mathematical Sciences, University of Montana. Email: \textbf{van.magnan@umontana.edu}}
\and
Abhishek Methuku\thanks{Department of Mathematics, ETH Z\"urich,
Switzerland. Research supported by SNSF grant 200021\_196965.
Email: \mbox{\textbf{
abhishekmethuku@gmail.com}}} 
}
\date{\vspace{-21pt}}

\title{Tight general bounds for the extremal numbers of 0--1 matrices}

\begin{document}

\maketitle

\begin{abstract}
    A zero-one matrix $M$ is said to contain another zero-one matrix $A$ if we can delete some rows and columns of $M$ and replace some $1$-entries with $0$-entries such that the resulting matrix is $A$. The extremal number of $A$, denoted $\ex(n,A)$, is the maximum number of $1$-entries that an $n\times n$ zero-one matrix can have without containing $A$. The systematic study of this function for various patterns $A$ goes back to the work of F\"uredi and Hajnal from 1992, and the field has many connections to other areas of mathematics and theoretical computer science. The problem has been particularly extensively studied for so-called acyclic matrices, but very little is known about the general case (that is, the case where $A$ is not necessarily acyclic). We prove the first asymptotically tight general result by showing that if $A$ has at most~$t$ $1$-entries in every row, then $\ex(n,A)\leq n^{2-1/t+o(1)}$. This verifies a conjecture of Methuku and Tomon.

    Our result also provides the first tight general bound for the extremal number of vertex-ordered graphs with interval chromatic number $2$, generalizing a celebrated result of F\"uredi, and Alon, Krivelevich and Sudakov about the (unordered) extremal number of bipartite graphs with  maximum degree $t$ in one of the vertex classes.
\end{abstract}

\section{Introduction}

One of the most central problems in extremal graph theory is concerned with estimating the maximum number of edges that a graph on a given number of vertices can have without containing some given graph as a subgraph. Formally, the \emph{extremal number} (also known as the Tur\'an number) of a graph $H$, denoted as $\ex(n,H)$, is the maximum number of edges that a graph on $n$ vertices can have if it does not contain $H$ as a subgraph. The first result on this topic was obtained by Mantel \cite{mantel1907solution} in 1907, when he determined the value of $\ex(n,H)$ when $H$ is the complete graph on three vertices. In 1941, Tur\'an \cite{turan1941extremal} extended this to the case where $H$ is an arbitrary complete graph. The celebrated Erd\H os--Stone--Simonovits theorem~\cite{erdos1946structure,erdos1966limit} asserts that $\ex(n,H)=\left(1-\frac{1}{\chi(H)-1}+o(1)\right)\binom{n}{2}$, which determines the asymptotics of $\ex(n,H)$ whenever the chromatic number $\chi(H)$ of $H$ is at least $3$. However, this result only gives the weak estimate $o(n^2)$ for the extremal number of bipartite graphs. A slightly better upper bound can be obtained using the K\H ov\'ari--S\'os--Tur\'an theorem \cite{kHovari1954problem} which states that $\ex(n,K_{s,t})=O(n^{2-1/s})$. Getting good estimates for the extremal number of bipartite graphs is an important and notoriously challenging problem. The difficulty is well reflected by the fact that the order of magnitude of $\ex(n,H)$ is unknown even in very simple cases such as when $H$ is $K_{4,4}$, the complete bipartite graph with $4$ vertices on each side, or $C_8$, the cycle of length~$8$. One of the few tight general results is the following theorem of F\"uredi~\cite{furedi1991turan}. A different proof of this result was later found by Alon, Krivelevich and Sudakov~\cite{alon2003turan} as one of the first applications of the celebrated dependent random choice method.

\begin{theorem}[F\"uredi \cite{furedi1991turan}, Alon--Krivelevich--Sudakov \cite{alon2003turan}] \label{thm:furediAKS}
    Let $H$ be a bipartite graph with maximum degree~$t$ in one side of the bipartition. Then $\ex(n,H)=O(n^{2-1/t})$.
\end{theorem}

This result is tight when $H=K_{t,r}$ for some $r$ much larger than $t$ (see \cite{kollar1996norm,alon1999norm,bukh2021extremal}), but there have been recent improvements in the case where $H$ is $K_{t,t}$-free: see \cite{conlon2021extremal,janzer2018improved,sudakov2020turan}.

\subsection{Zero-one matrices}

In the 90's, motivated by various applications to problems in other areas of mathematics (some of which we will discuss shortly), researchers started developing an analogous extremal theory for zero-one matrices.
For zero-one matrices $A$ and $M$, we say that $M$ \textit{contains} $A$ if, by deleting some rows and columns from $M$, and possibly turning some of its $1$-entries to $0$-entries, we can obtain $A$. The \textit{weight} of a zero-one matrix $M$, denoted $w(M)$, is the number of $1$-entries in $M$. The \textit{extremal number} of a zero-one matrix $A$, denoted $\ex(n,A)$, is the maximum possible weight of an $n\times n$ zero-one matrix that does not contain $A$.

One of the first results on the topic was obtained by F\"uredi \cite{furedi1990maximum} in 1990 who determined the order of magnitude of the extremal number of a certain $2\times 3$ matrix, and used it to give a $O(n\log n)$ bound for the number of unit distances in a convex $n$-gon, thereby making significant progress on an old problem of Erd\H os and Moser \cite{erdosmoser}. This bound is still the best known, though the implicit constant has been improved \cite{aggarwal2015unit} (also using forbidden submatrix theory). A more systematic study of extremal numbers of zero-one matrices was initiated by F\"uredi and Hajnal \cite{furedi1992matrices} in 1992. Since then, the extremal theory of zero-one matrices has been very successful at resolving problems in combinatorics, discrete and computational geometry, structural graph theory and the analysis of data structures. We refer the reader to the paper of Pettie and Tardos \cite{pettie2024extremal}, and the surveys of Tardos~\cite{tardos2018extremal, MR3967297} for an excellent overview of the extremal theory of zero-one matrices and its many applications.

There is a natural relationship between zero-one matrices and bipartite graphs. For a matrix $A$, we define $H_A$ to be the bipartite graph obtained by taking a vertex for each row and column of $A$, and taking an edge between $u$ and $v$ if and only if $A(u,v)=1$. It is easy to prove that for any zero-one matrix $A$, we have
\begin{equation}
    \ex(n,A)=\Omega(\ex(n,H_A)). \label{eqn:matrix and unordered}
\end{equation}

Interestingly, disproving a conjecture of F\"uredi and Hajnal~\cite{furedi1992matrices} in a very strong sense, Pach and Tardos~\cite{pach2006forbidden} showed that $\ex(n,A)$ can be much larger than $\ex(n,H_A)$. More precisely, they proved that there are matrices $A$ for which $H_A=C_{2k}$ (for arbitrary $k$) and yet $\ex(n,A)=\Omega(n^{4/3})$. Since $\ex(n,C_{2k})=O(n^{1+1/k})$ (see \cite{bondy74cycles}), this implies a huge gap between $\ex(n,A)$ and $\ex(n,H_A)$ for these matrices, and demonstrates that proving upper bounds for the extremal number of zero-one matrices is even more difficult than the corresponding problem for bipartite graphs.

When $H_A$ is a forest, we say that $A$ is an \emph{acyclic matrix}. The extremal numbers of acyclic matrices have been extensively studied. F\"uredi and Hajnal \cite{furedi1992matrices} conjectured that for any permutation matrix~$P$, we have $\ex(n,P)=O(n)$. Klazar \cite{klazar2000furedi} showed that this would imply the well-known Stanley--Wilf conjecture on the number of permutations without forbidden patterns, and Marcus and Tardos~\cite{marcus2004excluded} proved these conjectures. 
Another conjecture posed by F\"uredi and Hajnal asserted that for any acyclic zero-one matrix $A$, we have $\ex(n,A)=O(n\log n)$. This was disproved by Pettie \cite{pettie2011degrees}, who showed that there are acyclic zero-one matrices $A$ such that $\ex(n,A)=\Omega(n\log n\log\log n)$. Recently, Pettie and Tardos \cite{pettie2024extremal} found, for each positive integer $t$, an acyclic zero-one matrix $A_t$ such that $\ex(n,A_t)=\Omega(n(\log n)^t)$. Despite these advances, it is still very much unknown how large the extremal number of an acyclic zero-one matrix can be: there are no known acyclic zero-one matrices $A$ with $\ex(n,A)=\Omega(n^{1+\eps})$ for some $\eps>0$ but, strikingly, it is not even known whether there is an absolute constant $\eps>0$ such that $\ex(n,A)=O(n^{2-\eps})$ for all acyclic zero-one matrices $A$. 

Given our rather incomplete understanding of the theory even for acyclic zero-one matrices, it is unsurprising that very little is known about the general case, where $A$ is not necessarily acyclic.
To the best of our knowledge, the only known tight result for the extremal number of a zero-one matrix that is not acyclic concerns the $r\times t$ all-one matrix $A_{r,t}$ (naturally corresponding to the complete bipartite graph $K_{r,t}$) -- it is easy to show that $\ex(n,A_{r,t})=O(n^{2-1/t})$, and this is tight when $r$ is much larger than $t$ by equation (\ref{eqn:matrix and unordered}) and the known lower bounds for $\ex(n,K_{r,t})$. 

An important general result about extremal numbers of matrices that are not acyclic was proved by Methuku and Tomon \cite{methuku2022bipartite}. Say that a matrix $A$ is \textit{column-$t$-partite} (respectively, \emph{row-$t$-partite}) if it can be cut along its columns (respectively, rows) into $t$ submatrices such that every row (respectively, column) of each of these submatrices contains at most one $1$-entry.

\begin{theorem}[Methuku--Tomon \cite{methuku2022bipartite}]
    Let $t\geq 2$ and let $A$ be a column-$t$-partite zero-one matrix. Then $\ex(n,A)\leq n^{2-\frac{1}{t}+\frac{1}{t^2}+o(1)}$.
\end{theorem}

Since the $r\times t$ all-one matrix is column-$t$-partite, this is a fairly good estimate for large~$t$. Methuku and Tomon conjectured that their result can be strengthened in two ways: firstly, by improving the exponent to $2-1/t+o(1)$, and secondly, by significantly relaxing the column-$t$-partite assumption.

\begin{conjecture}[Methuku--Tomon \cite{methuku2022bipartite}] \label{conj:methukutomon}
    Let $t\geq 2$ and let $A$ be a zero-one matrix which contains at most~$t$ $1$-entries in each row. Then $\ex(n,A)\leq n^{2-\frac{1}{t}+o(1)}$.
\end{conjecture}

This conjecture is motivated by the aforementioned result of F\"uredi, and Alon, Krivelevich and Sudakov (Theorem \ref{thm:furediAKS}), and indeed (in view of equation (\ref{eqn:matrix and unordered})) it would be a direct generalization of that result, since zero-one matrices with at most $t$ $1$-entries in each row correspond to bipartite graphs with maximum degree at most $t$ in one side of the bipartition. As a partial result towards Conjecture \ref{conj:methukutomon}, Methuku and Tomon proved that $\ex(n,A)\leq n^{2-\frac{1}{t}+o(1)}$ when $A$ is both column-$t$-partite and row-$t$-partite. \medskip

In this paper we completely resolve Conjecture \ref{conj:methukutomon} as follows. 

\begin{theorem}\label{thm:generalwitho1}
    Let $t\geq 1$ and let $A$ be a zero-one matrix which contains at most $t$ $1$-entries in each row. Then $\ex(n,A)\leq n^{2-\frac{1}{t}+o(1)}$.
\end{theorem}

As we have discussed before, this is tight up to the $o(1)$ term. It is known~\cite{furedi1992matrices} that the $o(1)$ term is necessary in the case $t=1$; indeed, certain zero-one matrices $A$ with one $1$-entry in each row correspond to Davenport--Schinzel sequences \cite{davenport1965combinatorial} satisfying $\ex(n,A)=\omega(n)$. This suggests that perhaps the $o(1)$ term is necessary in general. On the other hand, we can show that for column-$t$-partite matrices, this error term is not needed when $t\geq 2$.

\begin{theorem}\label{thm:columntpartite}
    Let $t\geq 2$ and let $A$ be a column-$t$-partite zero-one matrix. Then $\ex(n,A)=O(n^{2-\frac{1}{t}})$.
\end{theorem}

Every matrix with at most one $1$-entry in each row is column-$1$-partite, so by our  discussion above, there are column-$1$-partite matrices with superlinear extremal numbers. Hence, Theorem~\ref{thm:columntpartite} reveals an intriguing difference between the cases $t=1$ and $t>1$. 

An interesting feature of our proof method is that, unlike the proof of Methuku and Tomon and many other important recent advances in the field \cite{korandi2019turan,kucheriya2023characterization}, it does not use a density increment argument. Instead, our proof employs a novel use of `blocks' of different sizes to construct an embedding of our forbidden matrix~$A$ (see Section~\ref{sec:outline} for an overview of our proof). We believe that this approach may have further applications in the extremal theory of zero-one matrices and ordered graphs.

\subsection{Ordered graphs}

In 2006, Pach and Tardos \cite{pach2006forbidden} initiated the systematic study of the extremal numbers of ordered graphs. An \textit{ordered graph} is a pair $(G,<)$ for which $G$ is a graph and $<$ is a total ordering of the vertex set of $G$ -- in what follows, we will sometimes abuse notation slightly and simply write $G$ for $(G,<)$. We say that $(G,<)$ contains $(H,<')$ as an \textit{ordered subgraph} if there exists an order-preserving embedding of $H$ into $G$. Analogously to the unordered setting, the extremal number of an ordered graph $H$, denoted by $\ex_<(n,H)$, is the maximum possible number of edges in an ordered graph on $n$ vertices that does not contain $H$ as an ordered subgraph. The \textit{interval chromatic number} of an ordered graph $H$, denoted by $\chi_<(H)$, is the minimum number of colours needed to colour the vertices of $H$ such that there are no edges within the colour classes, and each colour class is an interval with respect to the ordering on $V(H)$. Pach and Tardos \cite{pach2006forbidden} established an analogue of the Erd\H os--Stone--Simonovits theorem by proving that $\ex_<(n,H) = \left(1-\frac{1}{\chi_<(H)-1}+o(1) \right)\binom{n}{2}$ holds for any ordered graph $H$. This means that, much like in the unordered case, the asymptotic value of $\ex_<(n,H)$ is known unless $\chi_<(H)=2$. A graph $H$ is called \emph{ordered bipartite}\footnote{Note that with this definition an ordered bipartite graph is not quite the same as a bipartite graph with an arbitrary ordering of the vertices -- here, one part of the bipartition must completely precede the other in the ordering.} if $\chi_<(H)=2$.

There is a natural connection between ordered bipartite graphs and zero-one matrices: for an ordered bipartite graph $H$ with vertex classes $X,Y$ such that $\max(X)<\min(Y)$, we can define the matrix $A_H$ whose rows correspond to elements of $X$ (ordered according to the ordering of $V(H)$), whose columns correspond to elements of $Y$, and in which $A_H(x,y)=1$ if and only if $xy\in E(H)$. It is easy to see then that an ordered bipartite graph $G$ contains another ordered bipartite graph $H$ as an ordered subgraph if and only if $A_G$ contains $A_H$. Pach and Tardos established the following connection (which is much closer than the one established in (\ref{eqn:matrix and unordered}) between zero-one matrices and \emph{unordered graphs}) between the extremal numbers of $H$ and $A_H$.

\begin{theorem}[Pach--Tardos \cite{pach2006forbidden}] \label{thm:matrixandgraph}
    For any ordered bipartite graph $H$, we have
    \[
\ex(n,A_H)\leq \ex_<(2n,H)=O(\ex(2n,A_H)\log 2n),
\]
and if $\ex(n,A_H)=O(n^{c})$ for some $c>1$, then $\ex_<(n,H)=O(n^c)$.
\end{theorem}

Using this result, our Theorem \ref{thm:generalwitho1} immediately implies the following general bound on the extremal numbers of ordered bipartite graphs.

\begin{theorem}
    Let $H$ be an ordered bipartite graph with maximum degree at most $t$ in one side of the bipartition. Then $\ex_<(n,H)\leq n^{2-1/t+o(1)}$.
\end{theorem}

As a corollary, we obtain that if $H$ is an ordered even cycle (with interval chromatic number~$2$), then $\ex_<(n,H)\leq n^{3/2+o(1)}$. For some known results about the extremal number of certain \emph{families} of ordered even cycles, see \cite{gyHori2018turan}.

\medskip

\textbf{Organization of the paper.} In the next subsection, we give a detailed overview of the proof of Theorem~\ref{thm:generalwitho1}. The proof of Theorem \ref{thm:columntpartite} follows a similar strategy. In Section \ref{sec:proofs}, we prove both Theorem \ref{thm:generalwitho1} and Theorem \ref{thm:columntpartite}.

\subsection{Proof outline}
\label{sec:outline}

We now discuss some of the ideas used in our proof. For comparison, let us first recall the proof strategy of Methuku and Tomon~\cite{methuku2022bipartite} for their (weaker) bound of $n^{2-1/t+1/t^2+o(1)}$ in the special case of column-$t$-partite matrices. If $A$ is a fixed column-$t$-partite matrix, and $M$ is an $m\times n$ matrix which does not contain $A$, they divide $M$ into $k$ (horizontal) `blocks', i.e., submatrices of size $(m/k)\times n$, for some $k$. They show that if we cannot find a copy of $A$ where each row is coming from a different block, then one of the blocks must have large `density' in some sense -- more precisely, the number of copies of $K_{t,t}$ must be large in one of the blocks. They pass to that block (i.e., delete the rows corresponding to the other blocks) and repeat this process, obtaining a density increment at each step, eventually leading to a contradiction.

One can try to follow this strategy for general matrices $A$ which have at most $t$ $1$-entries in each row. However, without the column-$t$-partite condition, the density increment is too weak to obtain useful bounds. Hence, in this paper we will use a different argument. We will also repeatedly divide our large matrix $M$ into $k$ blocks. However, a key difference is that our argument is more `global' in the sense that we will not ignore the `deleted' rows in the blocks that we do not pass to -- in fact, they will be crucial to building our copy of $A$. Another key difference is coming from the way we choose the block we pass to: instead of passing to the block which is the `densest', we will select the block in a certain randomised way which works well together with classical dependent random choice-type arguments. Upon completion of this procedure we obtain a sequence of blocks, from which we build $A$.\medskip

To give more details about our proof, let us focus on the case $t=2$ (i.e., each row of $A$ has at most two $1$-entries), and assume that we are trying to prove a bound of $O(n^{1.51})$ for a matrix $A$ which has $20$ rows and $20$ columns. Let $n$ be large, and let $M$ be an $n\times n$ zero-one matrix of weight $n^{1.51}$. We perform the following procedure. First, we pick a row $r$ of $M$ uniformly at random, and look at the set of columns $C$ which have a $1$-entry in this row. (The row $r$ is fixed for the remainder of this procedure.) Let $s$ be a large but bounded number; for instance, $s=10^5$ works for the parameters above. This is the number of steps we will have in the process. Starting with our matrix $M_0=M$, we obtain $M_1$, $M_2$, \dots, $M_s$ as follows. Having defined $M_i$, we divide it into $k=n^{1/s}$ (horizontal) blocks of equal size, and $M_{i+1}$ is defined to be the block which contains the row $r$. Note that, after the last step, our matrix $M_s$ is the single row $r$.

For each pair of columns $c_1,c_2$ in $C$, we consider how the size of their common neighbourhood (i.e., the number of rows having a $1$-entry in both $c_1$ and $c_2$) changes during this process, i.e., we consider the numbers $N_j(c_1c_2)=|\{i:M_j(i,c_1)=M_j(i,c_2)=1\}|$ for $j=0,1,\dots,s$. Note that, for almost all pairs $\{c_1,c_2\}\in\binom{C}{2}$, $N_0(c_1c_2)$ is expected to be large (at least $n^{0.01})$, by our definition of $C$ using dependent random choice. However, $N_s(c_1c_2)$ is $1$ for all $\{c_1,c_2\}\in\binom{C}{2}$. We consider two types of steps for the columns $c_1,c_2$: `shrinking' steps, i.e., steps $j$ which have $N_{j}(c_1c_2)<N_{j-1}(c_1c_2)$, and `non-shrinking' steps, i.e., steps $j$ which have $N_{j}(c_1c_2)=N_{j-1}(c_1c_2)$. Note that we have $N_{j}(c_1c_2)<N_{j-1}(c_1c_2)$ if and only if more than one of the $k$ blocks of $M_{j-1}$ contain a row which has a $1$-entry in both of the columns $c_1$ and $c_2$. We consider two subcases for shrinking steps: a shrinking step for $c_1c_2$ is `above-shrinking' if there is a row in $M_{j-1}$ which has a $1$-entry in both of the columns $c_1$ and $c_2$, and this row is `above' the block containing $r$ (i.e., `above' the submatrix $M_{j})$, and `below-shrinking' otherwise (in which case there is a row in $M_{j-1}$ which has a $1$-entry in both of the columns $c_1$ and $c_2$, and this row is `below' the block containing $r$). 

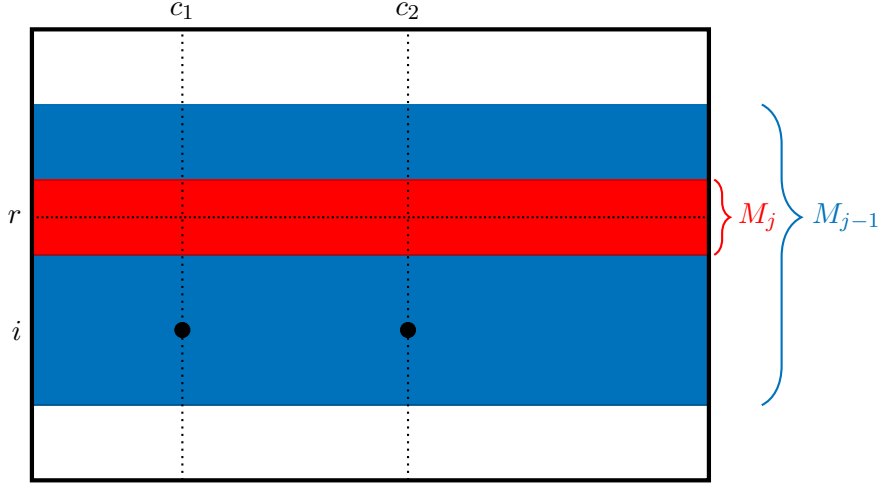
\begin{figure}[hbt!]
    \centering
    \begin{tikzpicture}

        \filldraw[fill=RoyalBlue, opacity=0.25] (1,-1) rectangle (10,-5);
        \draw[RoyalBlue] (1,-1) rectangle (10,-5);
        \filldraw[fill=red, opacity=0.4] (1,-2) rectangle (10,-3);
        \draw[red] (1,-2) rectangle (10,-3);
        
        \draw[dotted, thick] (3,-6)--(3,0);
        \draw[dotted, thick] (6,-6)--(6,0);

        \draw[densely dotted, thick] (1,-2.5) -- (10,-2.5);

        \filldraw[black] (3,-4) circle (0.1cm);
        \filldraw[black] (6,-4) circle (0.1cm);

        \node[anchor=east] (R) at (1,-2.5) {$r$};
        \node[anchor=east] (I) at (1,-4) {$i$};

        \node[anchor=south] at (3,0) {$c_1$};
        \node[anchor=south] at (6,0) {$c_2$};

        \draw[ultra thick] (1,0) rectangle (10,-6);

        \draw [thick, RoyalBlue,
	decorate, 
	decoration = {brace,
		raise=20pt,
		amplitude=15pt,
		aspect=0.375}] (10,-1) --  (10,-5)
node[pos=0.375,right=35pt,RoyalBlue]{$M_{j-1}$};

\draw [thick, red,
	decorate, 
	decoration = {brace,
		raise=2pt,
		amplitude=6pt,
		aspect=0.5}] (10,-2) --  (10,-3)
node[pos=0.5,right=7pt,red]{$M_j$};
    \end{tikzpicture}
    \caption{Step $j$ is `below-shrinking' for $c_1c_2$ because we have a row $i$ `below' $M_j$ that contains a $1$-entry in both of the columns $c_1$ and $c_2$.}
    \label{subtypefigure}
\end{figure}

To see why shrinking steps are useful, let us assume that we can find a set $C'=\{y_1,\dots,y_{20}\}$ of $20$ columns ($y_1<\dots<y_{20}$), together with $20$ steps $j_1,\dots,j_{20}\in[s]$ (where $j_1 < \ldots < j_{20}$), such that each of the steps $j_i$ is below-shrinking for all pairs of columns from $C'$.  For each $\ell\in[20]$, let $B_\ell$ be the submatrix of $M_{j_\ell-1}$ obtained by taking all rows below $M_{j_\ell}$. Note that the submatrices $B_\ell$ of $M$ are pairwise disjoint, and $B_{\ell}$ is located completely below $B_{\ell'}$ if $\ell<\ell'$. Moreover, for each $\ell\in[20]$, and each pair $c_1c_2$ from $C'$, since the step $j_\ell$ is below-shrinking for $c_1c_2$, we know that there is a row of $B_\ell$ which has a $1$-entry in both $c_1$ and $c_2$. However, these conditions easily imply that we can find a copy of $A$ in $M$ by using the columns $C'=\{y_1,\dots,y_{20}\}$, and embedding the $\ell$-th row of $A$ into $B_{21-\ell}$ appropriately (more precisely, if $A(\ell,a)=A(\ell,b)=1$, then the $\ell$-th row of the embedded copy of $A$ is an arbitrary row of $B_{21-\ell}$ which contains a $1$-entry in columns $y_a$ and $y_b$). (See Figure~\ref{fig:matrix} for an example of how shrinking steps can be used to construct an embedding of $A$ in $M$.)

Thus, it is enough to find a set $C'\subseteq C$ of $20$ columns, and a set of $20$ steps from $[s]$, such that each step is below-shrinking for each column-pair from $C'$. Similarly, it is enough to find a set $C'\subseteq C$ of $20$ columns, and a set of $20$ steps from $[s]$, such that each step is above-shrinking for each column-pair from $C'$. However, as noted above, if we look at a pair of columns, we `almost always' expect to have $N_0(c_1c_2)>n^{0.01}$, and we also know that $N_s(c_1c_2)=1$. Since we always divide into $k$ blocks and pass to one of the blocks in each shrinking step, $N_j(c_1c_2)$ is expected to shrink by at most a factor of $1/k=1/n^{1/s}$.  This can be made more precise using that, crucially, we always pass to the  block containing $r$, and $r$ is a uniformly random row in the common neighbourhood of $c_1$ and $c_2$. Thus, typically, we expect to have more than $0.01s$ shrinking steps -- in particular, for almost all pairs of columns, we expect more than $40$ shrinking steps. Hence, for almost all pairs of columns, we can find $20$ shrinking steps of the same subtype (above-shrinking or below-shrinking). Because this holds for almost all pairs of columns, we can find a large subset $C_1\subseteq C$ of columns such that each pair from $C_1$ has either $20$ above-shrinking steps or $20$ below-shrinking steps (by Turán's theorem). But then, since the number of `shrinking patterns' is bounded (as there are only $s=10^5$ steps, and we need to select $20$ of them to be above-shrinking or to be  below-shrinking), if $|C_1|$ is large enough, then by the multicolour Ramsey's theorem, we can find a large subset $C_2$  (of size more than $20$) of columns in $C_1$ such that every pair of columns in $C_2$ has the same shrinking pattern (i.e., there exist $20$ steps that are either all above-shrinking or all below-shrinking for every pair of columns in $C_2$), which finishes the proof by the observations above.

\begin{figure}[ht]
    \centering
    \begin{tikzpicture}

        \filldraw[fill=RoyalBlue, opacity=0.25] (0,-1) rectangle (10,-9.5);
        \draw[color=RoyalBlue] (0,-1) rectangle (10,-9.5);
        
        \filldraw[fill=red, opacity=0.4] (0,-2) rectangle (10,-7);
        \draw[color=red] (0,-2) rectangle (10,-7);
        
        \filldraw[fill=blue, opacity=0.3] (0,-2.5) rectangle (10,-4.5);
        \draw[color=blue] (0,-2.5) rectangle (10,-4.5);

        \draw[ultra thick] (0,0) rectangle (10,-11);

        \draw[densely dotted, thick] (0,-3)--(10,-3);

        \draw[dotted, thick] (3,-11)--(3,0);
        \draw[dotted, thick] (6,-11)--(6,0);
        \draw[dotted, thick] (8,-11)--(8,0);

        \filldraw[black] (3,-3) circle (0.1cm);
        \filldraw[black] (6,-3) circle (0.1cm);
        \filldraw[black] (8,-3) circle (0.1cm);

        \filldraw[black] (6,-7.4) circle (0.15cm);
        \filldraw[black] (8,-7.4) circle (0.15cm);

        \filldraw[black] (3,-4.9) circle (0.15cm);
        \filldraw[black] (8,-4.9) circle (0.15cm);

        \filldraw[black] (3,-9.9) circle (0.15cm);
        \filldraw[black] (6,-9.9) circle (0.15cm);

        \node[anchor=east] (A) at (0,-3) {$r$};

        \node[anchor=south] at (3,0) {$y_1$};
        \node[anchor=south] at (6,0) {$y_2$};
        \node[anchor=south] at (8,0) {$y_3$};

        \draw[blue, thick] (10.4,-2.5) -- (10.4,-4.5)
        node[pos=1.15] {$M_{j_3}$};
        \draw[blue, thick] (10.3,-2.5) --(10.5,-2.5);
        \draw[blue, thick] (10.3,-4.5) --(10.5,-4.5);

        \draw[red, thick] (10.8, -2) -- (10.8, -7)
        node[pos=1.06] {$M_{j_2}$};
        \draw[red, thick] (10.7,-2) --(10.9,-2);
        \draw[red, thick] (10.7,-7) --(10.9,-7);

        \draw[RoyalBlue, thick] (11.2,-1) -- (11.2,-9.5)
        node[pos=1.035] {$M_{j_1}$};
        \draw[RoyalBlue, thick] (11.1,-1)--(11.3,-1);
        \draw[RoyalBlue, thick] (11.1,-9.5)--(11.3,-9.5);
        
    \end{tikzpicture}
    \caption{The figure shows how to embed $A = 
\begin{bmatrix}
  1 & 0 & 1 \\
  0 & 1 & 1 \\
  1 & 1 & 0
\end{bmatrix}$ into $M$
    using 3 below-shrinking steps: we use step~$j_1$ for
    the column pair $y_1 y_2$ to embed the third row of $A$, 
    step $j_2$ for the column pair $y_2y_3$ to embed the second row of $A$, and step $j_3$ for the column pair $y_1 y_3$ to embed the first row of $A$.}
    \label{fig:matrix}
\end{figure}
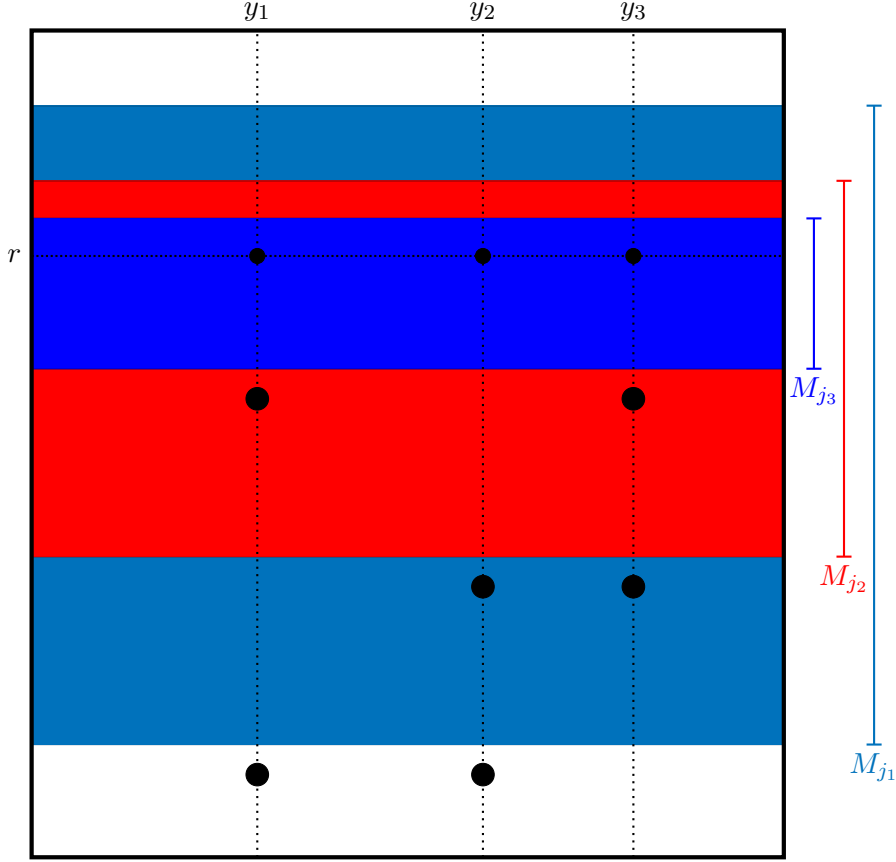

\section{Proofs of our results} \label{sec:proofs}
We now turn to the formal proofs of our results. We will use the following lemma to show that typically we have many shrinking steps.
\begin{lemma}\label{lemma_tree}
	Let $k$ and $m$ be non-negative integers with $k\geq 1$, and let $T$ be a rooted tree such that each vertex has at most $k$ children. Let us say that a vertex of $T$ is \emph{branching} if it has at least $2$ children. Then the number of (non-root) leaves of $T$ with at most $m$ branching ancestors is at most $k^m$.
\end{lemma}
\begin{proof}
	The result is established easily by induction on $m$, as follows. The case $m=0$ is clear, as if a leaf has zero branching ancestors then $T$ must be a path from the leaf to the root. Now assume that $m\geq 1$, and that the result holds for smaller values of $m$.

 For a given $m\geq 1$, we proceed by induction on $|T|$. The case $|T|= 1$ is clear. For $|T|>1$, if the root $r$ is not branching, then let $r'$ be its unique child, and we are done by applying induction to $T-r$ (rooted at $r'$). Otherwise let $r_1,\dots,r_\ell$ be the children of the root $r$, with $\ell\leq k$. Then $T-r$ splits up into $\ell$ trees $T_1,\dots,T_\ell$, rooted at $r_1,\dots,r_\ell$, respectively. Moreover, a vertex of $T_i$ has at most $m$ branching ancestors in $T$ if and only if it has at most $m-1$ branching ancestors in $T_i$. By induction, the number of such non-root leaves in $T_i$ is at most $k^{m-1}$. It follows that the number of non-root leaves in $T$ with at most $m$ branching ancestors is at most $\ell k^{m-1}\leq k^m$.
\end{proof}

We are now ready to prove Theorem \ref{thm:generalwitho1} in the following equivalent form.
	
\begin{theorem} \label{thm:general}
Let $t\geq 1$ be a positive integer. Let $\epsilon>0$, and let $A$ be a zero-one matrix such that each row of $A$ contains at most $t$ $1$-entries. Then $\ex(n,A)=O(n^{2-1/t+\epsilon})$.
\end{theorem}	

\begin{proof}
	Let $A$ have $a$ rows and $b$ columns. We may assume that $\epsilon<1/10$, $b\geq t$, and each row of $A$ contains exactly $t$ $1$-entries. Let $n$ be sufficiently large, let $M$ be an $n\times n$ zero-one matrix, and assume that $M$ has weight at least $n^{2-1/t+\epsilon}$. Let $s=\lceil 4a/\epsilon\rceil$, and let $k=\lceil n^{1/s}\rceil$. (Here, as described in the proof outline, $s$ denotes the number of steps we will perform, and $k$ is the number of (horizontal) blocks we divide our matrix $M$ into in each step.) Note that $s$ depends on $\epsilon$ and $A$ but not on $n$. Our goal is to show that $M$ must contain $A$. For convenience, we label the rows of $M$  by $[0,n-1]$ instead of $[n]$ (but we label the columns of $M$ by $[n]$). For each row index $i\in[0,n-1]$, let $\sigma(i)=(\sigma_1(i),\dots,\sigma_{s}(i))$ be the $k$-ary representation of $i$, that is, $\sigma_j(i)\in[0,k-1]$ for each $i\in[0,n-1]$ and $j\in[s]$, and we have
	$$i=\sum_{j=1}^{s}\sigma_{j}(i)k^{s-j}.$$ Note that the $k$-ary representation $\sigma(i)$ of row $i$ encodes for each step in the proof outline the block that the row is contained in.
	
We pick a row label $r\in[0,n-1]$ uniformly at random. Let $C$ be the set of column labels $c$ such that $M(r,c)=1$. For each $j\in[0,s]$, let $$R_j=\{i\in[0,n-1]:\sigma_\ell(i)=\sigma_\ell(r)\textnormal{ for all }\ell\in[j]\}.$$
Note that $R_j$ is the set of rows remaining after $j$ steps, i.e., the rows of the matrix $M_j$ in the proof outline, and we have $$[0,n-1]=R_0\supseteq R_1\supseteq \dots\supseteq R_s=\{r\}.$$ Furthermore, for each (unordered) $t$-set $e=\{c_1,\dots,c_t\}$ of distinct column labels from $C$, and for each $j\in[s]$, we define $\type_j(e)$ to be `shrinking'  if there is some row index $i\in R_{j-1}\setminus R_j$ such that $M(i,c_1)=\ldots=M(i,c_t)=1$. Otherwise, we define $\type_j(e)$ to be `non-shrinking'. (Note that these definitions agree with the ones mentioned in the proof outline.)
	
	Let a $t$-set $e$ of columns from $C$ be \emph{good} if there are at least $2a$ values of $j\in[s]$ such that $\type_j(e)$ is `shrinking', and let $e$ be \emph{bad} otherwise. Let $H$ denote the set of bad $t$-sets from $C$.  Clearly, $\EE[|C|]\geq n^{1-1/t+\epsilon}$ and hence
 $\EE[\binom{|C|}{t}]=\Omega(n^{t-1+t\epsilon})$.

\begin{claim} \label{claim:bad sets are rare}
	We have $\EE[|H|]=o(\EE[\binom{|C|}{t}])$.
\end{claim}
	
\begin{proof}
	Let a $t$-set of columns $e=\{c_1,\dots,c_t\}$ from $\binom{[n]}{t}$ be \emph{light} if $|\{i\in[0,n-1]:M(i,c_1)=\ldots=M(i,c_t)=1\}|\leq n^{\epsilon/2}$, and \emph{heavy} otherwise. Note that the expected number of light column $t$-sets in $C$ is at most $n^t\cdot (n^{\epsilon/2}/n)=n^{t-1+\epsilon/2}=o(\EE[\binom{|C|}{t}])$. Thus, it suffices to show that the expected number of column $t$-sets in $C$ which are both heavy and bad is $o(\EE[\binom{|C|}{t}])$. To show this, it is enough to prove that for any heavy $t$-set of columns $e$ from $\binom{[n]}{t}$, the conditional probability $\PP(e\textnormal{ is bad} \mid e\subseteq C)$ is $o(1)$.
	
	Let us fix a heavy $t$-set $e=\{c_1,\dots,c_t\}$, and let us write $I=\{i\in[0,n-1]:A(i,c_1)=\ldots=A(i,c_t)=1\}$. Note that, conditioned on $e\subseteq C$, $r$ is a uniformly random element of $I$. Let $V$ be the set of sequences $x=(x_1,\dots,x_p)$ of length at most $s$ such that there is some $i\in I$ with $(\sigma_1(i),\dots,\sigma_p(i))=x$. We define a rooted tree $T$ on $V$ by letting $x=(x_1,\dots,x_p)$ be the parent of $y=(y_1,\dots,y_q)$ precisely when $q=p+1$ and $(y_1,\dots,y_p)=(x_1,\dots,x_p)$ (i.e., $y$ is obtained from $x$ by extending the sequence by one step). In other words, the rooted tree corresponds to a poset on the initial segments of the $k$-ary representations of the row indices in $I$, where the elements of the poset are ordered by inclusion. Note that the root of the tree is the empty sequence, the leaves are $\sigma(i)$ for $i\in I$, and each vertex has at most $k$ children. Recall that a vertex of $T$ is branching if it has at least $2$ children. Notice that for all rows $i\in I$ and  for all $j\in[s]$, if $r=i$, then $\type_j(e)$ is `shrinking' if and only if $(\sigma_1(i),\dots,\sigma_{j-1}(i))$ is branching. Thus, if $r=i$, then $e$ is bad if and only if the corresponding leaf $\sigma(i)$ has at most $2a-1$ branching ancestors. Therefore, by Lemma~\ref{lemma_tree}, there are at most $k^{2a-1}$ choices of $r$ which make $e$ bad. Since $e$ is heavy, a uniformly random element of $I$ makes $e$ bad with probability at most $$k^{2a-1}/n^{\epsilon/2}\leq \lceil n^{\epsilon/(4a)}\rceil^{2a-1}/n^{\epsilon/2}=o(1),$$
	finishing the proof of the claim.
\end{proof}

	Let $N$ be a sufficiently large integer (namely, the multicolour hypergraph Ramsey number $N=r_{t}(b;2\binom{s}{a})$) such that every $2\binom{s}{a}$-colouring of the complete $t$-uniform hypergraph $K_N^{(t)}$ on $N$ vertices contains a monochromatic $K_b^{(t)}$.  Furthermore, let $L$ be a sufficiently large positive real number so that every $t$-uniform hypergraph with sufficiently many vertices and edge density more than $1-1/L$ contains a complete $t$-uniform subgraph $K_N^{(t)}$ on $N$ vertices (for example, we can choose $L=\binom{N}{t}$). Note that $N$ and $L$ do not depend on $n$. Using the claim above, we see that $\EE[\binom{|C|}{t}-L|H|]=\Omega(n^{t-1+t\epsilon})$. Therefore, we can fix a choice of $r$ such that we have $|C|=\Omega(n^{1-1/t+\epsilon})$ and $|H|< \frac{1}{L}\binom{|C|}{t}$. Then (if $n$ is sufficiently large), by the definition of $L$, we can find a subset $C_1\subseteq C$ such that $|C_1|=N$ and each $t$-set from $C_1$ is good (i.e., $C_1$ contains no $t$-set from $H$).
	
	Whenever $e=\{c_1,\dots,c_t\}$ is a $t$-set from $C_1$ and $j\in[s]$ is such that $\type_j(e)=$`shrinking', we let $\subtype_j(e)=\ \uparrow$ if there is some row index $i\in R_{j-1}\setminus R_j$ such that $i<\min R_j$ and $M(i,c_1)=\ldots=M(i,c_t)=1$. Otherwise, we let $\subtype_j(e)=\ \downarrow$. Note that in this case there is some row index $i\in R_{j-1}\setminus R_j$ such that $i>\max R_j$ and $M(i,c_1)=\ldots=M(i,c_t)=1$.
	For each $t$-set $e$ from $C_1$, we know that there are at least $2a$ choices of $j$ with type `shrinking', thus, we can find at least $a$ choices of $j$ with the same subtype ($\uparrow$ or $\downarrow$). For each $e$, choose some $z_e\in\{\uparrow,\downarrow\}$ and a subset $J_e\subseteq [s]$ of size $a$ such that $\subtype_j(e)=z_e$ for all $j\in J_e$ (and $\type_j(e)=$`shrinking' for all $j\in J_e$). Then $e\mapsto(z_e,J_e)$ gives a $2\binom{s}{a}$-colouring of the $t$-sets from $C_1$. Hence, by the definition of $N$, we can find $C_2\subseteq C_1$ of size $b$ such that each $t$-set from $C_2$ has the same colour -- i.e., there is some $z\in\{\uparrow,\downarrow\}$ and a subset $J\subseteq [s]$ of size $a$ such that for all $t$-sets $e$ from $C_2$ we have $z_e=z$ and $J_e=J$. In particular, $\subtype_j(e)=z$ for all $e\in\binom{C_2}{t}$ and all $j\in J$.
	
	We now show that this allows us to find a copy of $A$ on column set $C_2$. In what follows we will assume that $z=\ \downarrow$, as one can deal with the case $z=\ \uparrow$ similarly. Let us label the rows of $A$ by $[a]$ and the columns by $[b]$. Furthermore, let $C_2=\{c_1,\dots,c_b\}$ with $c_1<\dots<c_b$, and let $J=\{j_1,\dots,j_a\}$ with $j_1>\dots>j_a$. 
 For each $u\in[a]$, since $\subtype_{j_u}(\{c_y:y\in[b],A(u,y)=1\})=\ \downarrow$, there is some row $i_u\in R_{j_u-1}\setminus R_{j_u}$ of $M$ such that $i_u>\max R_{j_u}$ and $M(i_u,c_y)=1$ for all $y\in [b]$ with $A(u,y)=1$. Note that these conditions imply that $i_1<\dots<i_a$, and that $M(i_u,c_y)=1$ whenever $A(u,y)=1$. Thus, the rows $i_1,\dots,i_a$ and columns $c_1,\dots,c_b$ give a copy of $A$, finishing the proof.
\end{proof}

We now turn to the proof of Theorem \ref{thm:columntpartite}. The proof is quite similar to the proof of Theorem~\ref{thm:general}, so we will just highlight the differences.

\begin{proof}[Sketch of the proof of Theorem \ref{thm:columntpartite}]
    Let $M$ be an $n\times n$ zero-one matrix with weight $\omega(n^{2-1/t})$. It suffices to prove that if $n$ is sufficiently large, then $M$ contains $A$. Let $a$ be the number of rows in $A$. Consider an arbitrary partition of $A$ using vertical cuts into $t$ submatrices such that each of these submatrices has at most one $1$-entry in each row. We may assume without loss of generality that they all have precisely one $1$-entry in each row. Let $b_1,\dots,b_t$ be the number of columns of these submatrices in the natural order. One of the main differences compared to the proof of Theorem \ref{thm:general} is that we choose $k=2$ and $s=\lceil \log_2 n\rceil$. We define the random set $C$ of columns and the set $H$ of bad $t$-sets in $C$ as in the proof of Theorem~\ref{thm:general}. Now $\mathbb{E}[|C|]=\omega(n^{1-1/t})$ and hence $\mathbb{E}[\binom{|C|}{t}]=\omega(n^{t-1})$.
    
    We claim that, similarly to Claim \ref{claim:bad sets are rare}, we have $\mathbb{E}[|H|]\leq \frac{1}{4}\cdot \mathbb{E}[\binom{|C|}{t}]$ (for every sufficiently large $n$). Let a $t$-set of columns $e=\{c_1,\dots,c_t\}$ from $\binom{[n]}{t}$ be \emph{light} if $|\{i\in[0,n-1]:M(i,c_1)=\ldots=M(i,c_t)=1\}|\leq 5\cdot 2^{2a-1}$, and \emph{heavy} otherwise. Note that the expected number of light column $t$-sets in $C$ is at most $n^t\cdot (5\cdot 2^{2a-1}/n)=5\cdot 2^{2a-1}\cdot n^{t-1}=o(\EE[\binom{|C|}{t}])$. Thus, it suffices to show that the expected number of column $t$-sets which are both heavy and bad is at most $\frac{1}{5}\cdot \EE[\binom{|C|}{t}]$. To show this, it is enough to prove that for any heavy $t$-set of columns $e$ from $\binom{[n]}{t}$, the conditional probability $\PP(e\textnormal{ is bad} \mid e\subseteq C)$ is at most $1/5$. But an identical argument to the one in the proof of Claim \ref{claim:bad sets are rare} shows that this conditional probability is at most $\frac{k^{2a-1}}{5\cdot 2^{2a-1}}=1/5$.

    The rest of the argument is fairly different from the proof of Theorem \ref{thm:general}. Since $\mathbb{E}[|H|]\leq \frac{1}{4}\cdot \mathbb{E}[\binom{|C|}{t}]$, we have $\mathbb{E}[\binom{|C|}{t}-2|H|]\geq \frac{1}{2}\cdot \mathbb{E}[\binom{|C|}{t}]=\omega(n^{t-1})$. Hence, there exists an outcome such that $|C|=\omega(n^{1-1/t})$ and $|H|\leq \frac{1}{2}\cdot \binom{|C|}{t}$.

    For each $e=\{c_1,\dots,c_t\}\in \binom{C}{t}\setminus H$, let us define $z_e\in \{\uparrow,\downarrow\}$ and $J_e\subseteq [s]$ as in the proof of Theorem~\ref{thm:general}. Then $e\mapsto (z_e,J_e)$ is a $2\binom{s}{a}$-colouring of $\binom{C}{t}\setminus H$. It follows that there exist some $z\in \{\uparrow,\downarrow\}$ and $J\subseteq [s]$ of size $a$ such that for all at least $\frac{1}{2\binom{s}{a}}\cdot (\binom{|C|}{t}-|H|)\geq \frac{1}{4\binom{s}{a}}\cdot \binom{|C|}{t}$ sets $e\in \binom{C}{t}\setminus H$, we have $z_e=z$ and $J_e=J$. Now if $n$ is sufficiently large, then by an ordered analogue of the Erd\H os box theorem (see, e.g., Lemma 4 in \cite{methuku2022bipartite}), there exist sets $C_1,\dots,C_t\subseteq C$ of columns such that $|C_i|=b_i$ for every $i\in [t]$, every column in $C_i$ comes before every column in $C_{i+1}$ (for all $i\in [t-1]$), and for each $e=\{c_1,\dots,c_t\}$ which satisfies $c_i\in C_i$ for all $i\in [t]$, we have $z_e=z$ and $J_e=J$. Then, since $A$ is column-$t$-partite, we can embed $A$ into $M$ using the set $C_1\cup \dots \cup C_t$ of columns, similarly as in the proof of Theorem \ref{thm:general}.
\end{proof}

\noindent \textbf{Acknowledgements. } Our work on this paper started during the 13/14th Emléktábla Workshop, held in Vác, 13--18~August, 2023. We are grateful to Kristóf Bérczi,
Zoltán Lóránt Nagy and 
Balázs Patkós for organizing the workshop. We are also very grateful to Ervin Győri for many helpful discussions during the workshop.

\bibliographystyle{abbrv}
\bibliography{Bibliography}

\end{document}